\documentclass[12pt,reqno]{amsart}

\usepackage[usenames]{color} 

\usepackage{amsmath,amssymb}
\usepackage{hyperref}
\hypersetup{
    colorlinks=true,
    linkcolor=blue,
    filecolor=black,      
    urlcolor=black,
}
 
\newcommand{\B}{\mathcal{B}}

\newcommand{\F}{\mathcal{F}}
\newcommand{\PP}{\mathcal{P}}

\newcommand{\ot}{{\otimes}}
\newcommand{\op}{{\oplus}}

\newcommand{\ga}{\gamma}

\newcommand{\one}{\mathbf{1}}

\newcommand{\Z}{\mathbb{Z}}

\newcommand{\CC}{\mathcal{C}}

\DeclareMathOperator{\Diag}{Diag}
\DeclareMathOperator{\Sem}{Sem}
\DeclareMathOperator{\sVec}{sVec}
\DeclareMathOperator{\Rep}{Rep}
\DeclareMathOperator{\Hom}{Hom}

\newcommand{\mE}{\mathcal{E}}

\newcommand{\mZ}{\mathcal{Z}}

\newcommand{\mcd}{\mathcal{D}}
\newcommand{\mcg}{\mathcal{G}}

\newcommand{\Q}{\mathcal{Q}}
\theoremstyle{definition}
\newtheorem{thm}{Theorem}[section]
\newtheorem{rmk}{Remark}

\calclayout

\newcounter{commentcounter}

\title[]{On Realizing Modular Data}
\date{\today}
\author{Parsa Bonderson}
\email{parsab@microsoft.com}
\address{Microsoft Station Q,
    Santa Barbara, CA
    U.S.A.}
\author{Eric C. Rowell}
\email{rowell@math.tamu.edu}
\address{Department of Mathematics,
    Texas A\&M University,
    College Station, TX
    U.S.A.}
\author{Zhenghan Wang}
\email{zhenghwa@microsoft.com}
\address{Microsoft Station Q and Department of Mathematics,
    University of California,
    Santa Barbara, CA
    U.S.A.}
\begin{document}

\begin{abstract}

We use zesting and symmetry gauging of modular tensor categories to analyze some previously unrealized modular data obtained by Grossman and Izumi.  In one case we find all realizations and in the other we determine the form of possible realizations; in both cases all realizations can be obtained from quantum groups at roots of unity.

\end{abstract}
\thanks{}
\maketitle

\section{Introduction}

The modular data of a modular tensor category (MTC) is the most useful invariant of MTCs.  Recently, several iterations of new constructions inspired by the doubled Haagerup MTC produced intriguing new candidate modular data that have passed all known consistency conditions imposed on modular data from the full MTC structure \cite{EG}, with the latest construction presented in \cite{GI}.  These new candidate modular data lead to an obvious question: can we construct MTCs that realize these modular data?  The difficulty encountered in realizing MTCs with these candidate modular data suggests that some major constructions for MTCs have yet to be discovered.

As suggested in \cite{Bruillard16,CGPW,CSW}, zesting and symmetry gauging can be considered as a new construction of new MTCs from old ones.  In this short note, we point out that some of the candidate modular data in \cite{GI} can be realized by zesting and gauging of known modular tensor categories.  

\section{Realizing candidate modular data by zesting and gauging}

\subsection{Topological charge conjugation symmetry}

Modular data of a MTC was once conjectured to determine uniquely a MTC. However, this conjecture was recently disproven in \cite{Mignard-Schauenburg} by the construction of counterexamples. In retrospect, counterexamples to the conjecture were already implied by the results of \cite{Davydov}, which showed that there are some MTCs for which the map that sends an object $X$ to its dual $X^*$, which we will refer to as ``topological charge conjugation,'' is not a braided tensor auto-equivalence (topological symmetry). On the other hand, topological charge conjugation always preserves the modular data. It follows that the modular data cannot uniquely determine the MTC for the examples in which topological charge conjugation is not a braided tensor auto-equivalence. Even so, modular data is still a powerful and convenient invariant of MTCs.

\subsection{Grossman-Izumi Modular Data}
In \cite{GI}, candidate modular data is presented that generalizes the modular data of known Drinfeld centers of near-group fusion categories.
In detail, the input for the data consists of: two involutive metric groups $(G,q_1,\theta_1)$ and $(\Gamma,q_2,\theta_2)$, i.e.  $(G,q_1)$ and $(\Gamma,q_2)$ are metric groups, while $\theta_1$ and $\theta_2$ are, respectively, involutive automorphisms of $G$ and $\Gamma$ that preserve the quadratic forms $q_1$ and $q_2$.  Since $G$ and $\Gamma$ are necessarily Abelian, we will write them as additive, but the co-domain of $q_i$ is the multiplicative group $U(1)$.
Additionally, these must satisfy:
\begin{enumerate}
    \item The pre-metric groups obtained by restriction of $q_i$ to the $\theta_i$ fixed points of $G$ and $\Gamma$ coincide, i.e. $(G^{\theta_1},q_1)\cong (\Gamma^{\theta_2},q_2)=:(K,q_0)$.
    \item $\mathcal{G}(q_1)=\frac{1}{\sqrt{|G|}}\sum_{g\in G}q_1(g)=-\mathcal{G}(q_2)=\frac{1}{\sqrt{|\Gamma|}}\sum_{\gamma\in \Gamma}q_2(\gamma)$.
\end{enumerate}

Given such $(G,q_1,\theta_1),(\Gamma,q_2,\theta_2)$, one furthermore chooses $G_*\subset G$ and $\Gamma_*\subset \Gamma$ so that $G=K\sqcup G_*\sqcup \theta_1(G_*)$ and $\Gamma = K\sqcup \Gamma_*\sqcup \theta_2(\Gamma_*)$, where $K=G^{\theta_1}\cong\Gamma^{\theta_2}$.  The label set for the modular data is $$J:=K\sqcup( K\times {\pi})\sqcup G_*\sqcup \Gamma_*.$$  In particular the rank of the modular data is $|K|+\frac{|G|+|\Gamma|}{2}$.

The non-degenerate bicharacter associated with $(G,q_1)$ and $(\Gamma,q_2)$ will be denoted $B_i(g,h)=\frac{q_i(g)q_i(h)}{q_i(g+h)}$ for $i = 1$ and $2$, respectively, and we identify $K$ and $K\times\{\pi\}$ with the appropriate subgroup of both $G$ and $\Gamma$ when computing these values.  
Define constants $a=1/\sqrt{|G|}$ and $b=1/\sqrt{|\Gamma|}$.  The $S$ matrix has the following block structure, where $k,k^\prime$ range over $K$, $\kappa, \kappa^\prime$ range over $K\times\{\pi\}$, $g,g^\prime$ range over $G_*$, and $\gamma, \gamma^\prime$ range over $\Gamma_*$:
\begin{equation}
\label{s-matrix}
S = \left[
\begin{smallmatrix} \frac{a-b}{2}{B_1(k,k^\prime)} & \frac{a+b}{2}{B_1(k,\kappa^\prime)} & a{B_1(k,g^\prime)} & b{B_2(k,\gamma^\prime)}\\
\frac{a+b}{2} {B_1(\kappa,k^\prime)}& \frac{a-b}{2} {B_1(\kappa,\kappa^\prime)} &a{B_1(\kappa,g^\prime)}&-b{B_2(\kappa,\gamma^\prime)}
\\a{B_1(g,k^\prime)} & a{B_1(g,\kappa^\prime)}&a{(B_1(g,g^\prime)+B_1(\theta_1(g),g^\prime))} &0\\
b{B_2(\gamma,k^\prime)}& -b{B_2(\gamma,\kappa^\prime)}&0& -b\left(B_2(\gamma,\gamma^\prime)+B_2(\theta_2(\gamma),\gamma^\prime)\right)
\end{smallmatrix}
\right]
\end{equation}
The $T$ matrix has the form:
\begin{equation}
T = \Diag[q_0(k),q_0(\kappa),q_1(g),g_2(\gamma)].
\end{equation}

In general, it is an open question whether this data is realized, i.e., if there is a modular category with this $S$ and $T$ matrices. Assuming a realization exists, it is clear that the objects labeled by $K$ are invertible (have dimension $1$), the objects in $K\times\{\pi\}$ have dimension $\frac{a+b}{a-b}$, the objects labeled by $G_*$ have dimension $\frac{2a}{a-b}$, and the objects labeled by $\Gamma_*$ have dimension $\frac{2b}{a-b}$.  Thus, there are 4 distinct dimensions, generically.  Indeed, unless $9|G|=|\Gamma|$ the objects labeled by $K$ are the only invertible objects, and the fusion rules of this pointed subcategory are the same as the group operation in $K$, i.e., it is the pointed ribbon fusion category associated with the pre-metric group $(K,q_0)$.

\subsection{Zesting to realize case $G=\Z_2\times\Z_2$ or $G=\Z_4$ and $\Gamma=G\times \Z_3$}
At the October 2018 BIRS workshop on fusion categories and subfactors, Izumi presented \cite{izumibanff} the particular example $G=\Z_2\times\Z_2$ of their construction of candidate modular data and asked if a categorification is known. He pointed out that any such categorification could not be a Drinfeld center as the (multiplicative) central charge is not $1$.  The second and third authors recognized and pointed out the similarities between the presented data and that of the rank=$10$ Drinfeld center of the $1/2E_6$ theory $\mathcal{Z}(\mathcal{E})$ in \cite{HRW}.  We checked later that all 16 zestings \cite{Bruillard16} of $\mathcal{Z}(\mathcal{E})$ appear as Grossman-Izumi modular data, which also appeared in \cite{GI}.  We present the details of zesting here.  The case of zesting spin modular categories by the $\sVec$ subcategory afforded by the distinguished fermion is worked out in \cite{Bruillard16}, and the general case is found in \cite{DGPRZ}.

When $G=\Z_2\times \Z_2$ with $\theta_1(x,y)=(y,x)$ or $G=\Z_4$ with $\theta_1(x)=-x$ it is clear that $K\cong \Z_2$ in either case.  The possible $\theta_1$-invariant forms on $\Z_2\times \Z_2$ are $q_1^f(x,y)=(-1)^{x^2+xy+y^2}$, $q_1^{tc}(x,y)=(-1)^{xy}$ and  $q_1^s(x,y)=i^{\pm(x^2+y^2)}$, corresponding to the 3 fermion theory 3F, the toric code theory TC, and $(\Sem^{\boxtimes 2})^{\pm 1}$, respectively.  The corresponding Gauss sums are $-1,1$ and $\pm i$.  For $G=\Z_4$ there are also $4$ possible $\theta_1$-invariant forms: $q_1^r(x):=\zeta_8^{rx^2}$ with $r$ odd, corresponding to the $4$ distinct $\Z_4$ theories, with Gauss sums $\zeta_8^r$.  In all cases the form on $K\cong\Z_2$ is $q_0(x)=(-1)^x$ corresponding to the pre-metric group $\sVec$.  

As in \cite{GI}, the smallest interesting case is $\Gamma=G\times \Z_3$ with $G$ as above.  In order to fulfil $G^{\theta_1}\cong \Gamma^{\theta_2}$, we clearly should take $\theta_2(A,z)=(\theta_1(A),-z)$ where $A\in G$.  From this we can already see that $|G_*|=1$ and $|\Gamma_*|=5$.  For $G=\Z_2\times \Z_2$ we can choose $G_*=\{(1,0)\}$ and $\Gamma_*=\{(1,0,1),(1,0,0),(0,1,1),(1,1,1),(0,0,1)\}$, and for $G=\Z_4$ we take $G_*=\{1\}$ and $\Gamma_*=\{(1,0),(1,1),(0,1),(1,2),(2,1)\}$. In all cases, the object corresponding to the non-trivial element $\psi\in K$ is a fermion.  In particular, each category must be $\Z_2$-graded, with the objects $X$ satisfying $B_i(X,\psi)=1$ forming the trivial component.  In all cases, this trivial component is either $PSU(2)_{10}$ or its complex conjugate (which can be seen to be distinct, see \cite{Bruillard17}).
We go through the cases:
\begin{enumerate}
    \item If $q_1(x,y)=(-1)^{xy}$ then $q_2(x,y,z)=i^{\pm (x^2+y^2)}\omega^{\pm z^2}$ where $\omega=e^{i 2\pi /3}$.  These cases are recognized as $\mZ(\mE)$ and its complex conjugate: one sees that the (multiplicative) central charge is $1$.
    \item If $q_1(x,y)=(-1)^{x^2+xy+y^2}$ we similarly get two possible choices: $q_2(x,y,z)=i^{\pm (x^2+y^2)}\omega^{\mp z^2}$ where $\omega=e^{i 2\pi /3}$. 
    
\end{enumerate}

Conclusion: these 16 distinct modular data are realized by the rank $10$ minimal modular extensions of $PSU(2)_{10}$ and its complex conjugate.  

\subsection{Gauging to realize case $G=\Z_4\times\Z_4$, $\Gamma=\Z_{16}\times\Z_2$}

\subsubsection{Grossman-Izumi data for $G=\Z_4\times \Z_4$ and $\Gamma=\Z_{16}\times\Z_2$.} We follow the paper \cite[Section 3.2.3]{GI}.
In this case we start with 

\begin{equation}\label{thetas}
    \theta_1(x,y)=(y,x), \quad \theta_2(x,y)=(3x+8y,x+y).
\end{equation} It is clear that $G^{\theta_1}=\langle (1,1)\rangle\cong\Z_4$ and $\Gamma^{\theta_2}=\langle (4,1)\rangle \cong\Z_4$.  We assume that $q_2(x,y)=\zeta_{32}^{rx^2}i^{-ry^2}$, with $r\in\{\pm 1,\pm 3\}$, as these preserve $\theta_2$. Notice that the Gauss sum $\mcg(q_2)$ for $r=\pm 3$ is $-1$ whereas for $r=\pm 1$ we have $+1$.  Since $q_2(4,1)=q_2(12,1)=i^r$, this gives us a further constraint on $q_1$.  The possible $q_1$ and values of $r$ are determined essentially as follows:

\begin{enumerate}
    \item First consider the form $q_1(x,y)=i^{sxy}$ on $G=\Z_4\times \Z_4$ for $s$ odd.  Since $q_1(1,1)=q_1(3,3)=i^s$, we see that $s\equiv r\pmod{4}$ as $q_1(1,1)=q_2(4,1)$.  The Gauss sums are always $\mcg(q_1)=1$, so the $\mcg(q_1)=-\mcg(q_2)$ condition above shows we must take $r=\pm 3$.
    \item Next we consider $q_1(x,y)=i^{s(x^2+xy+y^2)}$ for $s$ odd.  The considerations as above again give $q_1(1,1)=i^{-s}$, so that $s\equiv-r\pmod{4}$.  The Gauss sums $\mcg(q_1)=1$ for all $s$, so again we must choose $r=\pm 3$.
    \item The case $q_1(x,y)=\zeta_8^{r(x^2+y^2)}$ gives $\mcg(q_1)=-i$ for all $r$, so there is no compatible choice of $r$.
    
\end{enumerate}  Grossman and Izumi \cite{GI} have verified that that such categories exist, at least for the choices $q_1(x,y)=i^{\pm 3 xy}$ with $r=\pm 3$ using \cite{Iz02}.  For the sake of definiteness we will choose $r=3$, so that $q_2(x,y)=\zeta_{32}^{3x^2}i^{y^2}$. There are two compatible choices of $q_1$, namely a hyperbolic form $q_1^h(x,y)=i^{-xy}$ and an elliptic form
$q_1^e(x,y)=i^{(x^2+xy+y^2)}$.  The other choices of $r$ correspond to complex conjugation.

In all cases we obtain rank $28$ modular data with labels $$J:=K\sqcup( K\times {\pi})\sqcup G_*\sqcup \Gamma_*.$$ as above.  Here we may take $K=\{(0,0),(1,1),(2,2),(3,3)\}=G^{\theta_1}$ or $$K=\{(0,0),(4,1),(8,0),(12,1)\}=\Gamma^{\theta_2}.$$  We take transversals:
$G_*=\{(1,0),(2,0),(2,1),(3,0),(3,1),(3,2)\}$
and
\begin{eqnarray*}
\Gamma_*=\{(1,0),(2,0),(3,0),(4,0),(5,0),(9,0),(10,0)\\  (0,1),(1,1),(2,1),(5,1),(6,1),(7,1),(13,1)\}.\end{eqnarray*}  When convenient we will decorate these labels with $g$ or $\gamma$ to indicate that they are in $G$ or $\Gamma$ respectively.
 The objects with labels in $K$ are invertible, while those with labels in $K\times \{\pi\}$ are $3+2\sqrt{2}$, the objects labeled by $G_*$ have dimension $4+2\sqrt{2}$ and those with labels in $\Gamma_*$ have dimension $2+2\sqrt{2}$.  

The pointed subcategory corresponding to labels in $K$ is a $\Z_4$ pre-metric group $\CC(\Z_4,q_0)$ with $q_0(x)=i^{-x^2}$.  
In particular there is a boson $b$ given by the element $(2,2)\in G^{\theta_2}$ or $(8,0)\in\Gamma^{\theta_2}$, which can be condensed to give a modular category $[\CC_{\Z_2}]_0$; that is, the trivial component of the $\Z_2$-de-equivariantization by $\Rep(\Z_2)\cong\langle b\rangle$.   This is equivalent to the \emph{modularization} \cite{Brug} of the centralizer of the category generated by the boson, i.e., $[\langle b \rangle ^\prime]_{\Z_2}$.  
The goal of this section is to prove the following, where $\overline{\F}$ denotes the complex conjugate of $\F$.

\begin{thm}
Let $\mathbf{G}:=(G,q_1^\epsilon,\theta_1)$ and $\mathbf{\Gamma}:=(\Gamma,q_2,\theta_2)$ be involutive metric groups with 
$G=\Z_4\times\Z_4$ and $\Gamma=\Z_{16}\times \Z_2$; $q_1^h(x,y)=i^{-xy}$,  
$q_1^e(x,y)=i^{(x^2+xy+y^2)}$ and $q_2(x,y)=\zeta_{32}^{3x^2}i^{y^2}$; and $\theta_1$ and $\theta_2$ as in (\ref{thetas}).
If $\CC$ is a modular category with $S$ and $T$ matrices as constructed above from $(\mathbf{G},\mathbf{\Gamma})$, then $[\CC_{\Z_2}]_0\cong \overline{\Sem}\boxtimes \B$ where $\B$ is either $\overline{PSU(3)_5}$ obtained from $\overline{PSU(3)_5}$ by the Galois automorphism $q\rightarrow -q$.
\end{thm}

\begin{proof}
The first step is to determine $\langle b\rangle^\prime$, where $b$ is the order $2$ element of $K$, i.e. $(2,2)\in G$ or $(8,0)\in \Gamma$.  By \cite{Brug}, the simple objects $X$ that centralize $b$ are precisely those with $\tilde{S}_{b,X}=d_X$, where $\tilde{S}$ is the $S$-matrix renormalized so that $\tilde{S}_{0,0}=1$ and $d_X$ is the dimension of the object labeled by $X$.  From the form of the $S$ matrix in Eqn.~(\ref{s-matrix}) we see that  these correspond to those $X\in J$ such that $B_i^\epsilon(b,X)=1$, where $B_1^h,B_1^e$ is the form obtained from $q_1^g$/$q_1^e$ and $B_2$ is obtained from $q_2$, depending on whether the label $X$ is in $G$ or $\Gamma$.  Since $B_1^\epsilon((2,2),(x,y))=(-1)^{x+y+2}$ for either choice $\epsilon\in\{h,e\}$ and $B_2((8,0),(x,y))=(-1)^x$ We find that these are:
\begin{enumerate}
    \item $K$
    \item $K\times \{\pi\}$
    \item $\{(2,0)_g,(3,1)_g\}\subset G_*$ and
    \item $\{(0, 1)_\ga, (2,1)_\ga, (6, 1)_\ga,(4, 0)_\ga,(2, 0)_\ga,(10, 0)_\ga\}\subset \Gamma_*$.
\end{enumerate}
In particular we find that the rank of $\mcd:=\langle b\rangle^\prime$ is $4+4+2+6=16$, and $\mcd$ is $\Z_2$-graded, inherited from the $\Z_4$ grading on $\CC$.  Since we have the $S$-matrix we can simply apply the Verlinde formula to determine the fusion rules and compute the fusion rules etc. for $\mcd$ directly.  However, for future work we prefer to obtain the result more economically, so we take a more finessed approach.

Next we will show that $\mcd_{\Z_2}$ has a $\boxtimes$-factorization into two modular categories. Since the pointed subcategory with labels in $K$ is $\CC(\Z_4,q_0)\subset \langle b\rangle^\prime$ and all other objects have non-integral dimension it is clear that the pointed subcategory of $\mcd_{\Z_2}$ is $\CC(\Z_4,q_0)_{\Z_2}$ which has rank $2$. In particular $\mcd_{\Z_2}$ is $\Z_2$-graded. Moreover, since modularization is a ribbon functor $F_b$ \cite{Brug} the non-trivial object $z:=F_b(1,1)$ has  $S_{z,z}=S_{(1,1),(1,1)}=-1$ and $\theta_z=\theta_{(1,1)}=q_1^\epsilon((1,1))=-i$ for $\epsilon\in\{e,h\}$.  We can then conclude that $z$ is (conjugate to) a semion, i.e., $\overline{\Sem}\cong\langle z\rangle\subset\mcd_{\Z_2}$ is the pointed subcategory. Thus $\mcd\cong\overline{\Sem}\boxtimes\B$ for some modular category $\B$.  

We now observe that $\B$ is the trivial component of the $\Z_2$-grading on $\mcd_{\Z_2}$, obtained as the $\Z_2$-de-equivariantization of $\CC_{ad}\subset \mcd$.  Now since $\CC_{ad}=\CC_{pt}^\prime$ and $\CC_{pt}\cong\CC(\Z_4,q_0)$, we employ the same technique above showing that a simple object $X\in\CC_{ad}$ if and only if $B_i^\epsilon(a,X)=1$ where the tensor generator $a\in\CC_{pt}$ is labeled by $(1,1)\in G$ or $(4,1)\in \Gamma$.  These are: $$J_{ad}=\{(0,0),(2,2), (0,0,\pi),(2,2,\pi),(3,1)_g, (2,1)_\gamma,(6,1)_\gamma,(4,0)_\gamma\}.$$ Let us denote the corresponding objects by $\one,b,X_1,X_2,Y,Z_1,Z_2,Z_3$ respectively.  The twists are as follows:
$\theta_\one=\theta_b=\theta_{X_i}=1$, $\theta_Y=i$, $\theta_{Z_1}=\theta_{Z_2}=e^{i 5\pi /4}$, and $\theta_{Z_3}=-1$.

In order to finish, we must determine the action of $b$ on this rank $8$ category.  Since $\dim(b\otimes X)=\dim(X)$ we immediately see that:
 $b\otimes Y\cong Y $ and $b\otimes X_1\cong X_2$ and of course $b\ot b\cong \one$. The usual yoga of de-equivariantization implies that $F_b(Y)=Y_1\oplus Y_2$ where $Y_i$ are simple objects of dimension $\dim(Y)/2=2+\sqrt{2}$ and $F_b(X_1)=F_b(X_2)=W$ is a simple object of dimension $3+2\sqrt{2}$.   We deduce from the twists of $Z_1,Z_2$ and $Z_3$ above that $b\otimes Z_3\cong Z_3$, so that $F_b(Z_3)=U_1\oplus U_2$ with each $U_i$ simple of dimension $\dim(Z_3)/2=1+\sqrt{2}$. Finally, we use the Verlinde formula to show that $b\ot Z_1\cong Z_2$, which implies that $F_b(Z_1)=F_b(Z_2)=V$ is a simple object of dimension $2+2\sqrt{2}$.
 Thus $\B$ is a rank $7$ modular category, with simple objects and data as in Table~\ref{Bdata}.
\begin{table}[h]
    \caption{$\B$ data}
    \label{Bdata}
	\centering\begin{tabular}{|c|c|c|c|c|}
		\hline
	$X$ & $d_X$ & $\theta_X$ & $J_{ad}$& $\CC_{ad}$\\
		\hline\hline
	$\one$ & $1$ & $1$ & $(0,0), (2,2)$& $\one,b$\\
		\hline
	$Y_i$ & $2+\sqrt{2}$ & $i$ & $(3,1)_g$& $Y$\\
		\hline
		$W$ & $3+2\sqrt{2}$ & $1$ & $(0,0,\pi),(2,2,\pi)$& $X_1,X_2$\\
		\hline
	$U_i$ & $1+\sqrt{2}$ & $-1$ & $(4,0)_\gamma$ & $Z_3$\\
	\hline
	$V$ & $2+2\sqrt{2}$ &$e^{i 5\pi /4}$ & $(2,1)_\gamma,(6,1)_\gamma$ & $Z_1,Z_2$\\
		\hline
	\end{tabular}
\end{table}

 We must work out the fusion rules for $\B$.  For this we use the Verlinde formula to determine the fusion rules of $\CC_{ad}$, and then use the fact that $\Hom_{\B}(F(x_1),F(x_2))\cong \Hom_{\CC_{ad}}(x_1,x_2\ot (\one\oplus b))$ from \cite{M2} and some multiplicity/dimension arguments to determine the fusion rules.  One can calculate that the fusion rules for $\CC_{ad}$ are identical for the two choices of $q_1^\epsilon$, so we may consider both cases simultaneously (see Remark \ref{rmk1} below).  Table~\ref{fusiontable} show some of the relevant calculations.
 
 \begin{table}[h]   \caption{Fusion rules of $\B$ from condensation of $\CC_{ad}$}
     \centering
     \begin{tabular}{|c|c|}\hline
     $\CC_{ad}$ & $\B$ \\
     \hline\hline
     $ Z_1\ot X_1\cong X_1 \oplus X_2 \oplus Y \oplus Z_2 \oplus Z_3 $   & $V\ot W\cong 2W\oplus V\oplus Y_1\oplus Y_2\oplus U_1\oplus U_2$\\\hline
     $Z_3^{\ot 2}\cong \one\op b\op Y\op Z_1\op Z_2\op Z_3$ & $(U_1\op U_2)^{\ot 2}\cong 2\one\op Y_1\op Y_2\op 2V\op U_1\op U_2$\\\hline
     $Y^{\ot 2}\cong \one\op b\op 2X_1\op 2X_2\op Y\op Z_1\op Z_2\op Z_3$ & $(Y_1\op Y_2)^{\ot 2}\cong 2\one\op 4W\op Y_1\op Y_2\op 2V\op U_1\op U_2$\\\hline
     $Z_3\ot Z_1\cong X_1\op X_2\op Y\op Z_3$ & $(U_1\op U_2)\ot V\cong 2W\op Y_1\op Y_2\op U_1\op U_2$\\\hline
   
     \end{tabular}
     \label{fusiontable}
 \end{table}
 Determining the fusion rules is a somewhat tedious sequence of calculations, essentially working out sufficiently many such rules until all are determined.  To illustrate this, we will show that the $Y_i$ and $U_i$ are non-self-dual and determine $Y_1\ot Y_2$ and $U_1\ot U_2$.
 
By the 2nd row of Table~\ref{fusiontable} we see that $U_1\ot U_2\cong \one\oplus V$ by multiplicity/dimension counting, hence $U_1\cong U_2^*$. Moreover $V\subset U_1\ot U_2$ implies that $U_i\subset U_i\ot V$ so the 4th and 1st rows of the table gives us
 $U_1\ot V\cong W\op U_1\op Y_j$ and $U_2\ot V\cong W\op U_2\op Y_k$ where $j\neq k$.  (Since there is labeling ambiguity between $Y_1$ and $Y_2$ we may choose $Y_j=Y_1$ and $Y_k=Y_2$.)  On the other hand, since $V^*\cong V$, we have $(U_1\ot V)^*\cong (U_2\ot V)$, so that $Y_1\cong Y_2^*$.
 A similar multiplicity/dimension calculation using the 4th row of the table now implies that $Y_1\ot Y_2\cong \one\op W\op V$.

Continuing in this way we eventually arrive at a complete set of fusion rules.  Ordering the simple objects in $\B$ as $[\one,W,Y_1,Y_2,V,U_1,U_2]$ the fusion rules can be determined from the fusion matrix for $Y_1$:
$$N_{Y_1}:=\left[ \begin {array}{ccccccc} 0&0&0&1&0&0&0\\ \noalign{\medskip}0&1&
1&1&1&1&0\\ \noalign{\medskip}1&1&0&0&1&0&0\\ \noalign{\medskip}0&1&1&0
&0&0&1\\ \noalign{\medskip}0&1&0&1&1&0&1\\ \noalign{\medskip}0&0&1&0&1
&0&0\\ \noalign{\medskip}0&1&0&0&0&1&0\end {array} \right]
.$$
  
  Now we claim that these fusion rules are the same as those of $PSU(3)_5$.  To see this one simply observes that the matrix $N_{Y_1}$ can be obtained from $N_\Lambda$ below by a permutation of columns/rows.  
  
  We can now adapt the main theorem of \cite{KazWenzl} to show that this is enough to show that $\B$ is obtained from ${PSU(3)_5}$ by Galois conjugation.  The argument is as follows: $PSU(3)_5$ is the trivial component of the $\Z_3$-grading on $SU(3)_5$, so that $\B\boxtimes \PP(\Z_3)$ and $SU(3)_5$ have the same fusion rules, where $\PP(\Z_3)$ is a pointed modular category with fusion rules like $\Z_3$.  The main result of \cite{KazWenzl} now implies that $\B\boxtimes \CC(\Z_3,q)$  is equivalent to a category $\F$ that is obtained from $SU(3)_5$ by twisting the associativity (by a $3$rd root of unity, in this case) and applying a Galois automorphism of $\Q(e^{i \pi /8})$. Now the associativity twisting is trivial on the trivial component of $SU(3)_5$, so $\F_0\cong \B$ is obtained from $PSU(3)_5$ by Galois conjugation.  
  
Only 4 of the 8 Galois conjugates of $PSU(3)_5$ have positive dimensions, two of which are known to be unitary. Further observe that the modular category $PSU(3)_5$ and its Galois conjugates have no fusion subcategories, and are \emph{unpointed} in the terminology of \cite{Nikshych-braidings} and therefore have exactly two braidings by \emph{loc. cit.} Corollary 4.8, which are reverses of each other.  Since there are no invertible objects, there is a unique spherical structure.  Comparing twists, we see that only the two choices $q=e^{- i \pi /8}$ and $q=e^{i 7\pi /8}$ are possible.
\end{proof}
\begin{rmk}\label{rmk1}
\begin{enumerate}
\item Although we did not use it in our proof, the $\Z_4$-grading on $\CC$ can be determined. We have $\CC_0=\CC_{ad}$ so that $J_{ad}=J_0$ while the simple objects in $\CC_2$ are those in $\mcd\setminus\B$ i.e., $$J_2=\{(1,1),(3,3),(1,1,\pi),(3,3,\pi),(2,0)_g,(0,1)_\gamma,(2,0)_\gamma,(10,0)_\gamma\}.$$
The other two components $\CC_1$ and $\CC_3$ are dual to each other, and so we have (where duals are vertically aligned):

$$J_1=\{(1,0)_g,(3,2)_g,(1,0)_\gamma,(5,0)_\gamma,(9,0)_\gamma,(7,1)_\gamma\},$$
$$J_3=\{(3,0)_g,(2,1)_g,(13,1)_\gamma,(1,1)_\gamma,(5,1)_\gamma,(3,0)_\gamma\}.$$
    \item We have uniquely determined the category $\B$ as obtained from the quantum group $U_q\mathfrak{sl}_3$ up to two choices of $q$: $e^{-i \pi /8}$ and $e^{i 7\pi /8}$.  It is conceivable that these two choices lead to equivalent categories, as $PSU(3)_5$ may only depend on $q^2$.  The first choice is known to be unitary, while the second is pseudo-unitary but unitarity is open.
    \item Observe that $q_1^e$ and $q_1^h$ take the same values on $K\times G$ (and $G\times K$), so that the $S$ and $T$ matrices for the two choices $q_1^e$ and $q_1^h$ are identical except for the objects with labels in $G_*$. This does not have any effect on the proof above as $q_1^h$ and $q_1^e$ coincide for $(3,1)_g$ and $(2,0)_g$.  
\end{enumerate}

\end{rmk}

The unitary modular category $SU(3)_5$ contains a pointed subcategory conjugate to $SU(3)_1$ whose complement is $PSU(3)_5$.  The non-trivial simple objects in $PSU(3)_5$ have highest weights $$\{(1,0,-1), (2,-1,-1),(1,1,-2),(2,0,-2),(3,-1,-2),(2,1,-3)\}.$$  In terms of the fundamental weights $\varpi_1=\frac{1}{3}(2,-1,-1)$ and $\varpi_2=\frac{1}{3}(1,1,-2)$ the coordinates are $[1,1],[3,0],[0,3],[2,2],[4,1]$ and $[1,4]$. Let us label them $\Upsilon,\Lambda,\Lambda^*,\Omega,\Xi,\Xi^*$ respectively.
The basic data for $PSU(3)_5$ are found in Table~\ref{psu35data}.

\begin{table}[h]
\caption{Basic Data}
\centering
\begin{tabular}{|c|c|c|c|}
\hline
&$x_1\varpi_1+x_1\varpi_2$ &$d_A$ & $\theta_A$  \\
\hline 
$\mathbf{1}$ & $[0,0]$ & $1$ & $1$ \\
\hline
$\Lambda,\Lambda^*$ &$[3,0],[0,3]$ & $2+\sqrt{2}$ & $-i$ \\
\hline
$\Upsilon$  & $[1,1]$ & $2+2\sqrt{2}$ & $e^{i 3\pi /4}$ \\
\hline
 $\Xi,\Xi^*$ & $[4,1],[1,4]$ & $1+\sqrt{2}$ & $-1$\\
\hline
 $\Omega$& $[2,2]$ & $3+2\sqrt{2}$ & $1$ \\
\hline 
\end{tabular}\label{psu35data}\end{table}

If we order the objects: $[\mathbf{1},\Lambda,\Lambda^*,\Upsilon,\Xi,\Xi^*,\Omega]$ the fusion rules are easily determined using standard Lie theory.  We find that the fusion matrix for $\Lambda$ is:
$$N_\Lambda:=\left[ \begin {array}{ccccccc} 0&0&1&0&0&0&0\\ \noalign{\medskip}1&0&0
&1&0&0&1\\ \noalign{\medskip}0&1&0&0&0&1&1\\ \noalign{\medskip}0&0&1&1
&0&1&1\\ \noalign{\medskip}0&1&0&1&0&0&0\\ \noalign{\medskip}0&0&0&0&1
&0&1\\ \noalign{\medskip}0&1&1&1&1&0&1\end {array} \right]$$  The full set of fusion rules can be derived from those of $\Lambda$.  This can be seen as follows: since $N_\Lambda$ has $7$ distinct eigenvalues, any matrix that commutes with $N_\Lambda$ is a polynomial in $N_\Lambda$.  In particular each fusion matrix $N_x$ is a polynomial in $N_\Lambda$, and one column/row of $N_x$ is determined.

\bibliographystyle{plain}

\end{document}